\documentclass[12pt]{amsart}

\usepackage{amsmath,amsthm,amssymb,amsfonts,verbatim}
\usepackage{graphicx}
\usepackage[colorlinks, linkcolor=blue,  anchorcolor=blue,
            citecolor=blue
            ]{hyperref}
\usepackage[hmargin=1.2in,vmargin=1.2in]{geometry}

\title[Optimal LRC via elliptic function fields]{The group structures of automorphism groups of elliptic function fields over finite fields and their applications to optimal locally repairable codes}
\author{Liming Ma}\address{School of Mathematical Sciences, University of Science and Technology of China, Hefei 230026}\email{lmma20@ustc.edu.cn}
\author{Chaoping Xing} \address{School of Electronics, Information and Electric Engineering, Shanghai Jiao Tong University,
China 200240}\email{xingcp@sjtu.edu.cn}
\date{}

\newtheorem{lemma}{Lemma}[section]
\newtheorem{theorem}[lemma]{Theorem}
\newtheorem{cor}[lemma]{Corollary}

\newtheorem{prop}[lemma]{Proposition}

\newtheorem{defn}{Definition}

\theoremstyle{remark}
\newtheorem{rmk}{Remark}

\renewcommand{\epsilon}{\varepsilon}
\renewcommand{\le}{\leqslant}
\renewcommand{\ge}{\geqslant}

\def\Gal{{\rm Gal}}

\def\ZZ{\mathbb{Z}}
\def\PP{\mathbb{P}}

\def\F{\mathbb{F}}

\def \fE {\mathfrak{E}}

\def \mL {\mathcal{L}}

\def \mP {\mathcal{P}}

\def \Xi {{X^{[i]}}}

\newcommand{\Ga}{\alpha}
\newcommand{\Gb}{\beta}

\newcommand{\Gs}{\sigma}

\def \bc {{\bf c}}

\def \bo {{\bf 0}}

\def\supp {{\rm supp }}

\def\Aut {{\rm Aut }}

\def\LRC {{\rm locally repairable code\ }}

\def\Gal{{\rm Gal}}

\begin{document}

\maketitle

\begin{abstract}
The automorphism group of an elliptic curve over an algebraically closed field is well known. However, for various applications in coding theory and cryptography, we usually need to apply automorphisms defined over a finite field. Although we believe that the automorphism group of an elliptic curve over a finite field is well known in the community, we could not find this in the literature. Nevertheless, in this paper we show the group structure of the automorphism group of an elliptic curve over a finite field. More importantly,  we characterize subgroups and abelian subgroups of  the automorphism group of an elliptic curve over a finite field.

Despite of theoretical interest on this topic, our research is largely motivated by constructions of optimal locally repairable codes. The first research to make use of automorphism group of function fields to construct optimal locally repairable codes was given in a paper \cite{JMX20} where automorphism group of a projective line was employed. The idea was further generated to an elliptic curve in \cite{MX19} where only automorphisms fixing the point at infinity were used. Because there are at most $24$ automorphisms of an elliptic curve fixing the point at infinity, the locality of optimal locally repairable codes from this construction is upper bounded by $23$. One of the main motivation to study  subgroups and abelian subgroups of  the automorphism group of an elliptic curve over a finite field is to remove the constraints on locality.  
\end{abstract}

\section{Introduction}\label{sec:1}
Let $q$ be a power of a prime $p$. Let $\F_q$ be the finite field with $q$ elements.
Let $F$ be an algebraic function field of one variable with the full constant field $\F_q$.
Let $\Aut(F/\F_q)$ denote  the automorphism group of function field $F$ over $\F_q$.
The study of automorphism groups of function fields is very interesting in both theory and applications  \cite{BTV17,GK09,GK19,GX15,GXY18,LMX19,MX19tams}.

The automorphism group of rational function field $\F_q(x)$ over $\F_q$ is the projective linear group and denoted by ${\rm PGL}_2(\F_q)$.
The subgroup structures of ${\rm PGL}_2(\F_q)$ can be found from \cite{HKT08,JMX20} and the generators of Galois subfields of the rational function field are characterized in \cite{H20}.
The subgroups of automorphism group $\Aut(\F_q(x)/\F_q)$ can be employed to construct optimal locally repairable codes \cite{TB14,JMX20,JKZ20}.

The Hermitian function field $H/\F_{q^2}$ is the well-known maximal function field and its automorphism group is the projective unitary group ${\rm PGU}(3,\F_{q^2})$ which is large and has many subgroups. Many new maximal function fields can be constructed from fixed subfields of various subgroups of automorphism group of the Hermitian function field \cite{BMXY13,GSX00,MX19,MZ20}. The automorphism group of Hermitian function field has been applied to construct locally repairable codes in \cite{BTV17}. 

Over an algebraic closed field $\bar{\F}_q$, the automorphism group of an elliptic function field is well known \cite[Theorem 11.94]{HKT08}.
However, we usually need the automorphism group of an elliptic function field over finite fields for applications in coding theory and cryptography.
Although we believe that the automorphism group of an elliptic curve over a finite field is well known in the community, we could not find this in the literature. Nevertheless, in this paper we show the group structure of the automorphism group of an elliptic curve over a finite field. More importantly,  we will characterize subgroups and abelian subgroups of  the automorphism group of an elliptic curve over a finite field. In particular, when the rational points of an elliptic curve form a cyclic group, we show that one can find all subgroups of the automorphism group of an elliptic curve over a finite field. For abelian subgroups, we only consider those generated by one translation and one automorphism fixing the point at infinity.  

The first part of this paper is devoted to determining the automorphism groups of elliptic function fields over finite fields and their subgroups. In the second part, we give an explicit construction of optimal locally repairable codes via various subgroups of automorphism groups of elliptic function fields over finite fields which generalize the idea given in \cite{LMX19}.

Despite of theoretical interest on this topic, our research on   automorphism groups of  elliptic curves is largely motivated by constructions of optimal locally repairable codes. The first research to make use of automorphism group of function fields to construct optimal locally repairable codes was given in a paper \cite{JMX20} where automorphism group of a projective line was employed. The idea was further generalized to an elliptic curve in \cite{MX19} where only automorphisms fixing the point at infinity were used. Because there are at most $24$ automorphisms of an elliptic curve fixing the point at infinity, the locality of optimal locally repairable codes from this construction is upper bounded by $23$. One of the main motivation to study  subgroups and abelian subgroups of  the automorphism group of an elliptic curve over a finite field is to remove the constraints on locality. It is worth to mention that the automorphism group of an elliptic curve over a finite field was also employed in \cite{BHHMV17} to construct locally repairable codes. However, the codes in \cite{BHHMV17} are not optimal. In addition, the authors make use of only translations in the paper  \cite{BHHMV17}. In summary,  only automorphisms fixing the point at infinity were used in \cite{JMX20}, while only translations were employed in \cite{BHHMV17}. In this paper we make use of subgroups  involving both translations and automorphisms fixing the point at infinity. The advantage of using subgroups  involving both translations and automorphisms fixing the point at infinity is that locality can be much larger than $23$.

Due to recent applications in distributed storage systems and cloud storage systems, locally repairable codes have attracted great attention of researchers \cite{HL07,GHSY12,PKLK12,SRKV13,FY14,PD14,TB14,TPD16,BTV17,LMX19b, MX20}.
For a locally repairable code $C$ of length $n$ with $k$ information symbols and locality $r$, it was proved in \cite{GHSY12} that the minimum distance $d(C)$ is upper bounded by
 \begin{equation}\label{eq:x1}
 d(C)\le n-k-\left\lceil \frac kr\right\rceil+2.
 \end{equation}
The bound \eqref{eq:x1} is called the Singleton-type bound for locally repairable codes and any locally repairable code achieving this bound is called optimal.

Construction of optimal  locally repairable codes is of both theoretical interest and practical importance. This is a challenging task and has attracted great attention in the last few years. In \cite{HCL},  pyramid codes are shown to be optimal locally repairable codes. In \cite{SRKV13},  Silberstein {\it et al}  proposed a two-level construction based on the Gabidulin codes combined with a single parity-check $(r+1,r)$ code. Another construction \cite{TPD16} used two layers of MDS codes, a Reed-Solomon code and a special $(r+1,r)$ MDS code. A common shortcoming of these constructions relates to the size of the code alphabet which in all these papers is an exponential function of the code length, complicating the implementation. In \cite{PKLK12}, there is a construction of optimal locally repairable codes with alphabet  size comparable to code length. However, its length is a specific value $n=\left\lceil \frac kr\right\rceil(r+1)$ and its rate is very close to $1$.

A remarkable breakthrough construction of optimal locally repairable codes was given Tamo and Barg in \cite{TB14}. This construction naturally generalizes Reed-Solomon construction which relies on the alphabet of cardinality comparable to the code length $n$. The idea behind the construction is very nice. The only shortcoming of this construction is restriction on locality $r$. This construction was extended via automorphism groups of rational function fields by Jin, Ma and Xing \cite{JMX20} and it turns out that there are more flexibility on locality and the code length can be up to $q+1$. One generalization is to recover multiple erasures \cite{CFXF19,CMST20,CXHF18,FF20}.

Based on the classical MDS conjecture, one should wonder if $q$-ary optimal  locally repairable codes can have length bigger than $q+1$. Surprisingly, it was shown in \cite{BHHMV17} that there exist $q$-ary optimal locally repairable codes of length exceeding  $q+1$.  For small localities $r=2,3,5,7,11$ or $23$, there exists an explicit construction of optimal locally repairable codes via automorphism groups of elliptic curves \cite{LMX19}. All the above constructions of optimal locally repairable codes can have large distances which are proportional to the alphabet size $q$.

For small distances, there exists many optimal locally repairable codes with length super-linear to the alphabet size $q$.
In \cite{LXY19}, Luo {\it et al.} provided a construction of optimal locally repairable codes with unbounded length for distance $d=3$ or $4$.
In \cite{GXY19}, Guruswami {\it et al.} proved that the length of optimal locally repairable codes is upper bounded by $O(dq^{3+\frac{4}{d-4}})$, and there exist optimal locally repairable codes with length $\Omega_r(q^{1+1/[(d-3)/2]})$ for distance $d\ge 5$. In \cite{GXY19,J19}, there exist optimal locally repairable codes with length $O(q^2)$ for distance $d=5,6$. In \cite{XY18}, there exist explicit constructions of optimal locally repairable codes with super-linear length for any constant distance $d\ge 7$.

As for large distances $d=O(n)$ and locality $r$ is a constant, the length $n$ of optimal locally repairable codes is upper bounded by $O(q)$ from \cite[Corollary 14]{GXY19}.
To our best knowledge,  the construction via elliptic curves is the few known family of optimal locally repairable codes with larger distance and length exceeding $q+1$ in the literature. However, the locality must be small, such as $2,3,5,7,11$ or $23$. Hence, it is still worth to provide a general construction of optimal locally repairable codes via automorphism groups of elliptic function fields as \cite{JMX20}. In particular, we shall provide an explicit construction of $q$-ary optimal locally repairable codes with more flexible locality $r$ and length up to $q+2\sqrt{q}-2r-1$ or even $q+2\sqrt{q}-r$ for locality $r=8$.

This paper is organized as follows. In Section \ref{sec:2}, we introduce some preliminaries for this paper such as  elliptic function fields over finite fields, the ramification theory of elliptic function fields, maximal elliptic curves and their automorphism groups, algebraic geometry codes and locally repairable codes.
In Section \ref{sec:3}, we determine the automorphism groups of elliptic function fields over finite fields and provide a characterization of their subgroups including abelian subgroups. In Section \ref{sec:4}, we present an explicit construction of optimal locally repairable codes via automorphism groups of elliptic function fields which is the generalization of \cite{LMX19}.

\section{Preliminaries}\label{sec:2}
In this section, we present some preliminaries on elliptic function fields over finite fields, the ramification theory of elliptic function fields, maximal elliptic curves and their automorphism groups, algebraic geometry codes and locally repairable codes.

\subsection{Elliptic function fields over finite fields}
By a curve, we will always mean a projective, smooth and absolutely irreducible algebraic curve.
An elliptic curve $\fE$ defined over a finite field $\F_q$ is given by
 a nonsingular Weierstrass equation
\begin{equation}\label{eq:3} y^2+a_1xy+a_3y=x^3+a_2x^2+a_4x+a_6,\end{equation}
where $a_i$ are elements of $\F_q$. The function field of $\fE$ is given by $\F_q(\fE)=\F_q(x,y)$. For convenience of this paper, we use the language of function fields rather than curves although we also refer to curves occasionally. The reader may refer to \cite{St09} and \cite{Si86} for the languages of elliptic function fields and elliptic curves, respectively.
Let $E$ be the elliptic function field $\F_q(\fE)=\F_q(x,y)$ defined above.
 The genus of $E$ is $g(E)=1$. Let $O$ be the common pole of $x$ and $y$ which is the unique infinity place of $E$.  Any solution $(a,b)\in \F_q^2$ to the Weierstrass equation \eqref{eq:3}  corresponds to a rational place of $E$ which is the unique common zero of $x-a$ and $y-b$. We denote by $\fE(\F_q)$ the set of rational points on $\fE$, i.e., $\fE(\F_q)=\{(\Ga,\Gb)\in\F_q^2:\;  \Gb^2+a_1\Ga\Gb+a_3\Gb=\Ga^3+a_2\Ga^2+a_4\Ga+a_6\}\cup\{O\}$.

Let $\mathbb{P}_E$ be the set of all places of $E$ and $\mathbb{P}^1_E=\{P\in \mathbb{P}_E:\deg(P)=1\}$ be the set of rational places of $E$. Then there is one-to-one correspondence between the set  $\mathbb{P}^1_E$ and $\fE(\F_q)$.
The divisor group of $E/\F_q$ is the free abelian group generated by the places in $\mathbb{P}_E$ and is denoted by Div$(E)$. The set of divisors of degree $0$ forms a subgroup of Div$(E)$ and is denoted by $\text{Div}^0(E)$.
The principal divisor of $z\in E^*$ is defined by $(z)=\sum_{P\in \mathbb{P}_E} \nu_P(z)P,$ where $\nu_P$ is the normalized discrete valuation with respect to $P$.
Two divisors of $E$ are called equivalent if there exists an element $z\in E^*$ such that $A=B+(z)$, which is denoted by $A\sim B$.
The group of principal divisors of $E$ is $\text{Princ}(E)=\{(x): x\in E^*\}$ and the factor group $\text{Cl}^0(E)=\text{Div}^0(E)/\text{Princ}(E)$ is called the group of divisor classes of degree zero.

From \cite[Proposition 6.1.7]{St09}, there is a group isomorphism between  $\mathbb{P}^1_E$ and $\text{Cl}^0(E)$ given by
$$\Phi:\begin{cases}\mathbb{P}^1_E\rightarrow \text{Cl}^0(E),\\ P\mapsto [P-O].\end{cases}$$
The group operation of $\mathbb{P}^1_E$ is defined by $P\oplus Q=R \Leftrightarrow P+Q\sim R+O$ for any $P,Q\in \mathbb{P}^1_E$.
Denote by $[2]P=P\oplus P$ and define $[m+1]P=[m]P\oplus P$ recursively.
In fact, $\mathbb{P}^1_E$ is an abelian group and the place $O$ is the zero element of the group $\mathbb{P}^1_E$.

\begin{lemma}\label{lem:2.1}
Let $E/\F_q$ be an elliptic function field and let $P,Q$ be two rational places of $E$. Then we have
$$P\sim Q \text{ if and only if } P=Q.$$
\end{lemma}
\begin{lemma}\label{lem:2.2}
Let $E/\F_q$ be an elliptic function field with $N(E)$ rational places. Then we have the following Hasse-Weil bound
$$|N(E)-q-1|\le 2\sqrt{q}.$$
\end{lemma}
An elliptic function field $E/\F_q$ is called maximal if its number of rational places attains the above Hasse-Weil bound, i.e., $N(E)=q+2\sqrt{q}+1$.

\subsection{Ramification theory of elliptic function fields}
Let $E/\F_q$ be a function field with the full constant field $\F_q$. Let $g(E)$ denote by the genus of $E$.
Let $D$ be a divisor of $E$. The Riemann-Roch space $$\mathcal{L}(D)=\{z\in E^*: (z)\ge -D\}\cup \{0\}$$
is a finite-dimensional vector space over $\F_q$ and its dimension is at least $\deg(D)-g(E)+1$ from Riemann's theorem  \cite[Theorem 1.4.17]{St09}. If $E$ is an elliptic function field and $\deg(D)\ge 1$, then its dimension is $\dim_{\F_q}\mL(D)=\deg(D)-g(E)+1=\deg(D)$.

Let $\Aut(E/\F_q)$ be the automorphism group of elliptic function field $E$ over $\F_q$, that is,
$$\Aut(E/\F_q)=\{\sigma: E\rightarrow E|\;  \sigma \text{ is an } \F_q \text{-automorphism of } E\}.$$
Let $G$ be a subgroup of $\Aut(E/\F_q)$. The corresponding fixed field of $G$ is defined by
$$E^{G}=\{z\in E: \Gs(z)=z \text{ for all } \Gs\in G\}.$$
From Galois theory, $E/E^{G}$ is a Galois extension with $\Gal(E/E^{G})=G.$
Let $F=E^{G}$. Then $\F_q$ is the full constant field of $F$ and 
the Hurwitz genus  formula \cite[Theorem 3.4.13]{St09} yields
\[2g(E)-2=[E:F](2g(F)-2)+\deg \text{Diff}(E/F),\]
where Diff$(E/F)$ is the different of $E/F$.
If there is a place $Q$ of $E$  with ramification index $e_Q>1$, then the different exponent of $Q$ is at least $d(Q)\ge e(Q)-1\ge 1$ from Dedekind's different theory \cite[Theorem 3.5.1]{St09} and hence $F$ is a rational function field.

\subsection{Maximal elliptic curves and their automorphism groups}
In order to construct algebraic geometry codes with good parameters, we usually need function fields over finite fields with many rational places, especially maximal function fields.

We say that two elliptic curves $\fE_1$ and $\fE_2$ over $\F_q$ are isogeny if there is a non-constant smooth $\F_q$-morphism from $\fE_1$ to $\fE_2$ that sends the zero of $\fE_1$ to the zero of $\fE_2$ (see \cite{Si86}). It is a well-known fact that two  elliptic curves $\fE_1$ and $\fE_2$ over $\F_q$ are isogenous if and only if they have the same number of rational points. More precisely speaking, we have the following result \cite{Wa69}.

\begin{lemma}\label{lem:2.3}
The isogeny classes of elliptic curves over $\F_q$ for $q=p^a$ are in one-to-one correspondence with the rational integers $t$ having $|t|\le 2\sqrt{q}$
and satisfying some one of the following conditions:
\begin{itemize}
\item[(i)] $(t,p)=1$;
\item[(ii)] If $a$ is even: $t=\pm 2\sqrt{q}$;
\item[(iii)] If $a$ is even and $p\not \equiv 1\pmod{3}$: $t=\pm \sqrt{q}$;
\item[(iv)] If $a$ is odd and $p=2$ or $3$: $t=\pm p^{\frac{a+1}{2}}$;
\item[(v)] If either (i) $a$ is odd or (ii) $a$ is even and $p\not \equiv 1\pmod{4}: t=0.$
\end{itemize}
Furthermore, an elliptic curve in the isogeny class corresponding to $t$ has $q+1+t$ rational points.
\end{lemma}

The group structure of abelian group consisting of all rational places of an elliptic curve $E$ over $\F_q$ can be given in the following proposition \cite[Theorem 3]{Ru87} and \cite[Theorem 9.97]{HKT08}.
\begin{prop}\label{prop:2.4}
Let $\F_q$ be the finite field with $q=p^s$ elements. Let $h=\prod_\ell \ell^{h_\ell}$ be a possible number of rational places of an elliptic curve $E$ of $\F_q$. Then all the possible groups $\PP_E^1$ are the following
\[\mathbb{Z}/p^{h_p}\ZZ\times \prod_{\ell \neq p}\left( \mathbb{Z}/\ell^{a_\ell}\ZZ\times \mathbb{Z}/\ell^{h_\ell-a_\ell}\ZZ\right)\]
with \begin{itemize}
 \item[(a)] In case (ii) of Lemma \ref{lem:2.3}: Each $a_\ell$ is equal to $h_\ell/2$, i.e, $\PP_E^1\cong \ZZ/(\sqrt{q}\pm1)\ZZ \times  \ZZ/(\sqrt{q}\pm1)\ZZ.$
\item[(b)] In other cases of Lemma \ref{lem:2.3}: $a_\ell$ is an arbitrary integer satisfying $0\le a_\ell \le \min\{\nu_\ell(q-1),[h_\ell/2]\}$.
In cases (iii) and (iv) of Lemma \ref{lem:2.3}: $\PP_E^1\cong \ZZ/h\ZZ.$
 In cases (v) of Lemma \ref{lem:2.3}: if $q\not\equiv -1(\text{mod } 4)$, then $\PP_E^1\cong \ZZ/(q+1)\ZZ$; otherwise, $\PP_E^1\cong \ZZ/(q+1)\ZZ$ or $\PP_E^1\cong \ZZ/2\ZZ \times \ZZ/\frac{q+1}{2}\ZZ$.
\end{itemize}
\end{prop}

Let $\mathfrak{E}/\F_q$ be an elliptic curve defined by the Weierstrass equation \eqref{eq:3}.
We denote by $\Aut(\fE)$ the set of automorphisms of elliptic curve $\fE$ over the algebraic closure $\bar{\F}_q$, i.e., every automorphism $\Gs\in\Aut(\fE)$ fixes the infinity place $O$. The following result can be found in \cite[Theorem 3.10.1]{Si86}.

\begin{prop}\label{prop:2.5} Let $\mathfrak{E}/\F_q$ be an elliptic curve with $j$-invariant $j(\fE)$.  Then the order of $\Aut(\fE)$  divides $24$. More precisely speaking, the order of
$\Aut(\fE)$ is given by the following list:
\begin{itemize}
\item[{\rm (i)}] $|\Aut(\fE)|=2$  if  $j(\fE)\neq 0, 1728$;
\item[{\rm (ii)}] $|\Aut(\fE)|=4$ if $j(\fE)=1728$ and char$(\F_q)\neq 2,3$;
\item[{\rm (iii)}] $|\Aut(\fE)|=6$ if $j(\fE)=0$ and char$(\F_q)\neq 2,3$;
\item[{\rm (iv)}] $|\Aut(\fE)|=12$ if $j(\fE)=0=1728$ and char$(\F_q)=3$;
\item[{\rm (v)}] $|\Aut(\fE)|=24$ if $j(\fE)=0=1728$ and char$(\F_q)=2$.
\end{itemize}
\end{prop}

Let $E/\F_q$ be the function field $\F_q(\fE)$ and denote by  $\Aut(E/\F_q)$ the automorphism group of $E$ fixed every element of $\F_q$.
Let $\Aut(E,O)$ be the set of $\F_q$-automorphisms of $E$ fixing the infinite place $O$.
Then $\Aut(E,O)$ is a subgroup of $\Aut(\fE)$ in which every automorphism is defined over $\F_q$, i.e., we have $\Aut(E,O)=\Aut(\fE) \cap \Aut(E/\F_q).$

By summarizing the results given in Sections $2$ and $3$ in \cite{LMX19}, we provide some examples of maximal elliptic curves with explicit automorphism groups as follows.

\begin{lemma}\label{lem:2.5}
For any even $a$, there exists a maximal elliptic function field $E$ over $\F_{2^a}$ defined by an equation $y^2+y=x^3+\Ga$ for some $\Ga\in\F_{2^a}$ such that $|\Aut(E,O)|=24$. Any automorphism $\Gs\in \Aut(E,O)$ is given by
\[\label{eq:6}\Gs(x)=u^2x+s,\quad \Gs(y)=u^3y+u^2sx+t,\]
where $u,s,t\in\F_{2^a}$ satisfy $u^3=1$, $s^4+s=0$ and $t^2+t+s^6=0$. Furthermore, for any divisor $d$ of $24$, there is a subgroup of
$\Aut(E,O)$ of order $d$.
\end{lemma}
\begin{proof}
Please refer to \cite[Lemma 15 and Lemma 9]{LMX19}.
\end{proof}

\begin{lemma}\label{lem:2.6}
For any even $a$, there exists a maximal elliptic function field $E$ over $\F_{3^a}$ defined by an equation $y^2=x^3+\Ga x$ with $-\Ga$ being a square in $\F_{3^a}^*$ such that $|\Aut(E,O)|=12$. Any automorphism $\Gs\in \Aut(E,O)$ is given by
\[\Gs(x)=u^2x+s,\quad \Gs(y)=u^3y,\]
where $u,s\in\F_{3^a}$ satisfy $u^4=1$, $s^3+\Ga s=0$.  Furthermore, for any divisor $d$ of $12$, there is a subgroup of
$\Aut(E,O)$ of order $d$.
\end{lemma}
\begin{proof}
Please refer to \cite[Lemma 16 and Lemma 10]{LMX19}.
\end{proof}

\begin{lemma}\label{lem:2.7}
Let $p\neq 2$ and $p\equiv 2\pmod{3}$ be an odd prime. For any even $a$, there  exists a maximal elliptic function field $E$ over $\F_{p^a}$ defined by an equation $y^2=x^3+ \theta^3$ for some $\theta\in\F_{p^a}^*$ such that $|\Aut(E,O)|=6$. Any automorphism $\Gs\in \Aut(E,O)$ is given by
\[\label{eq:xm2}\Gs(x)=u^2x,\quad \Gs(y)=u^3y,\]
where $u\in\F_q^*$ satisfies $u^6=1$.   Furthermore, for any divisor $d$ of $6$, there is a subgroup of
$\Aut(E,O)$ of order $d$.
\end{lemma}
\begin{proof}
Please refer to \cite[Lemma 17 and Lemma 11]{LMX19}.
\end{proof}

\begin{lemma}\label{lem:2.8}
Let $p\neq 3$ and $p\equiv 3\pmod{4}$ be a prime. Then for any even $a$, there  exists a maximal elliptic function field $E$ over $\F_{p^a}$ defined by an equation $y^2=x^3+ \theta^2 x$ for some $\theta\in\F_{p^a}^*$ such that $|\Aut(E,O)|=4$. Any automorphism $\Gs\in \Aut(E,O)$ is given by
\[\label{eq:xm2}\Gs(x)=u^2x,\quad \Gs(y)=u^3y,\]
where $u\in\F_q^*$ satisfies $u^4=1$.  Furthermore, for any divisor $d$ of $4$, there is a subgroup of
$\Aut(E,O)$ of order $d$.
\end{lemma}
\begin{proof}
Please refer to \cite[Lemma 16 and Lemma 12]{LMX19}.
\end{proof}

\subsection{Algebraic geometry codes}
For the construction of algebraic geometry codes, the reader may refer to \cite{St09} for more details.
Let $E/\F_q$ be a function field and let $\mP=\{P_1,\dots,P_n\}$ be a set of $n$ distinct rational places of $E$. For a divisor $D$ of $E$ with $0<\deg(D)<n$ and $\supp(D)\cap\mP=\emptyset$, the algebraic geometry code associated with $D$ and $\mP$ is defined by
\[C(\mP,D):=\{(f(P_1),\dots,f(P_n)): \; f\in\mL(D)\}.\]
Then $C(\mP,D)$ is an $[n,k,d]$-linear code with dimension $k=\deg(D)$ and minimum distance $d\ge n-\deg(D)$.

In order to construct optimal locally repairable codes with larger length, we need to remove the restriction $\supp(D)\cap\mP=\emptyset$.
Let $m_i=\nu_{P_i}(D)$ and choose a prime element $\pi_i$ of $P_i$ for each $i\in \{1,2,\cdots, n\}$.
The modified algebraic geometry code given in \cite{JMX20} is defined as follows
\[C(\mP,D):=\left\{\left((\pi_1^{m_1}f)(P_1),\dots,(\pi_n^{m_n}f)(P_n)\right):\; f\in\mL(D)\right\}.\]
Then $C(\mP,D)$ is an $[n,k,d]$-linear code with dimension $k=\deg(D)$ and minimum distance $d\ge n-\deg(D)$.
It is easy to see that the Hamming weight of the codeword $\left((\pi_1^{m_1}f)(P_1),\dots,(\pi_n^{m_n}f)(P_n)\right)$ is at least $n-\deg(D)$ for every nonzero function $f\in\mL(D)$. Let $I$ be the subset of $\{1,2,\dots,n\}$ such that $(\pi_i^{m_i}f)(P_i)=0$. Then we have $0\neq f\in\mL(D-\sum_{i\in I}P_i)$. Thus, $\deg(D-\sum_{i\in I}P_i)\ge 0$, i.e., $|I|\le \deg(D)$. The weight of this codeword is lower bounded by  $n-|I|\ge n-\deg(D)$.

Let $V$ be a subspace of $\mL(D)$, then a subcode of $C(\mP,D)$ can be defined by
\[C(\mP,V):=\{((\pi_1^{m_1}f)(P_1),\dots,(\pi_n^{m_n}f)(P_n)):\; f\in V\}.\]
The minimum distance of  $C(\mP,V)$ is still lower bounded by $n-\deg(D)$.

\subsection{Locally repairable codes}
Informally, a block code $C$ is said with locality $r$ if  every coordinate of any given codeword in $C$ can be recovered by accessing at most $r$ other coordinates of this codeword. The formal definition of a locally repairable code with locality $r$ is given as follows.
\begin{defn}
Let $C\subseteq \F_q^n$ be a $q$-ary block code of length $n$. For each $\Ga\in\F_q$ and $i\in \{1,2,\cdots, n\}$, define $C(i,\Ga):=\{\bc=(c_1,\dots,c_n)\in C\; :\; \; c_i=\Ga\}$. For a subset $I\subseteq \{1,2,\cdots, n\}\setminus \{i\}$, we denote by $C_{I}(i,\Ga)$ the projection of $C(i,\Ga)$ on $I$.
Then $C$ is called a locally repairable code with locality $r$ if, for every $i\in \{1,2,\cdots, n\}$, there exists a subset
$I_i\subseteq \{1,2,\cdots, n\}\setminus \{i\}$ with $|I_i|\le r$ such that  $C_{I_i}(i,\Ga)$ and $C_{I_i}(i,\Gb)$ are disjoint for any $\Ga\neq \Gb\in\F_q$.
\end{defn}
In this paper, we focus on linear codes. Thus, a $q$-ary \LRC of length $n$, dimension $k$, distance $d$ and locality $r$ is denoted as a $q$-ary $[n,k,d]$-\LRC with locality $r$. It was shown in \cite{GHSY12}  that the minimum distance $d(C)$ of $C$ is upper bounded by
 \[d(C)\le n-k-\left\lceil \frac kr\right\rceil+2, \]
which is now called the Singleton-type bound for locally repairable codes.  A code achieving this bound is referred to as an optimal locally repairable code.
We shall construct optimal locally repairable codes via automorphism groups of elliptic function fields over finite fields.

\section{Automorphism groups of elliptic function fields over finite fields and their subgroups}\label{sec:3}
Let $E$ be an elliptic function field defined over the finite field $\F_q$. Let $\bar{\F}_q$ be the algebraic closure of $\F_q$.
The automorphism group of $E\bar{\F}_q$ over $\bar{\F}_q$ is well known in the literature, such as in \cite[Theorem 11.9.4]{HKT08} and \cite[Theorem 10.5.1]{Si86}. As for various applications in coding theory and cryptography, we usually need automorphisms defined over the finite field $\F_q$. We believe that the automorphism group of $E$ over $\F_q$ is also known in the community, although we could not find a good reference.  Nevertheless,
in this section, we shall determine the group structure of the automorphism group of an elliptic function field $E$ over $\F_q$ for completeness and characterize their subgroups and abelian subgroups. Define
\[\Aut(E/\F_q)=\{\sigma: E\rightarrow E|\;\sigma \text{ is an } \F_q\text{-automorphism of } E \}.\]
In fact, the automorphism group $\Aut(E/\F_q)$ of the elliptic function field $E$ over $\F_q$ is a subgroup of the automorphism group of $\Aut(E\bar{\F}_q/\bar{\F}_q)$.

\subsection{Automorphism groups of elliptic function fields over finite fields}\label{subsec:3.1}
In this subsection, we give an explicit characterization of automorphism groups of elliptic function fields over finite fields via group action.
Let $\PP_E^1$ be the set of rational places of $E/\F_q$. Now we consider the group action of automorphism group $\Aut(E/\F_q)$ on the set of rational places $\PP_E^1$. Let $\sigma$ be an automorphism of $E$ over $\F_q$ and let $P$ be a rational place of $E$. From \cite[Lemma 3.5.2]{St09}, we have $\sigma(P)$ must be a rational place of $E$.

Let $\Aut(E,O)$ be the stabilizer of the infinite place $O$ under the group action of $\Aut(E/\F_q)$, i.e., $\Aut(E,O)=\{\Gs\in \Aut(E/\F_q):\Gs(O)=O\}$. It consists of all   automorphisms defined over $\F_q$ in the automorphism group $\Aut(\fE)$ of the corresponding elliptic curve $\fE$ given in Proposition \ref{prop:2.5}. The automorphism group $\Aut(E/\F_q)$ acts transitively on the set of rational places $\PP_E^1$, since the translation-by-$Q$ map $\tau_Q$ defined by $\tau_Q(P)=P\oplus Q$ induces an automorphism of the elliptic function field $E$ over $\F_q$ for each place $Q\in \PP_E^1$. Hence, the orbit of $O$ under the group action of  $\Aut(E/\F_q)$ is $\PP_E^1=\{\Gs(O):\Gs\in \Aut(E/\F_q)\}$. By the theory of group action, the order of the automorphism group $\Aut(E/\F_q)$ is given by $|\Aut(E/\F_q)|=|\PP_E^1| \cdot |\Aut(E,O)|.$ Let $T_E$ be the translation group $\{\tau_Q:\; Q\in \PP_E^1\}$ of the elliptic function field $E$, we have \[ |\Aut(E/\F_q)|=|T_E| \cdot |\Aut(E,O)|.\]  We have the following group structure of automorphism group $\Aut(E/\F_q)$ of the elliptic function field $E$ over $\F_q$.

\begin{theorem}\label{thm:3.1}
Let $E$ be an elliptic function field defined over $\F_q$. Let $\Aut(E,O)$ be the stabilizer of the infinite place $O$ with automorphisms defined over $\F_q$. Let $T_E$ be the translation group $\{\tau_Q:\; Q\in \PP_E^1\}$ of $E$. The automorphism group of elliptic function field $E$ over $\F_q$ is the semidirect product of the translation group $T_E$ and the stabilizer $\Aut(E,O)$ of the infinite place $O$, i.e., \[\Aut(E/\F_q)=T_E \rtimes \Aut(E,O). \] The group law of $\Aut(E/\F_q)$ is given by $(\tau_P\Ga)\cdot (\tau_Q\Gb)=\tau_{P\oplus \Ga(Q)}\cdot \Ga\Gb$ for any $\tau_P,\tau_Q\in T_E$ and $\Ga,\Gb\in \Aut(E,O)$.
\end{theorem}
\begin{proof}
For every automorphism $\sigma \in \Aut(E/\F_q)$, it is easy to verify that $\tau_{-\sigma(O)}  \sigma: E\rightarrow E$ is an automorphism of $E$ fixing $O$, since
$\tau_{-\sigma(\mathcal{O})} \sigma(O)=\sigma(O)\oplus (-\sigma(O))=O.$ Then we have $\tau_{-\sigma(O)} \sigma=\rho \in \Aut(E,O)$, i.e., $\sigma= \tau_{\sigma(O)}\rho\in \tau_{\sigma(O)} \Aut(E,O)\subseteq T_E \cdot \Aut(E,O)$. From the order of automorphism group, we have $\Aut(E/\F_q)=T_E \cdot \Aut(E,O)$.

It is easy to see that the intersection of $T_E$ and $\Aut(E,O)$ is the unique identity isomorphism, since the translation $\tau_Q$ maps $O$ to another rational place $Q$ of $E$ for $Q\neq O$. Hence,  every automorphism $\sigma \in \Aut(E/\F_q)$ can be uniquely written as a composition of a translation $\tau_{\sigma(O)}$ and an automorphism $\rho$ of $E$ fixing $O$.

We claim that the translation group $T_E$ is a normal subgroup of  $\Aut(E/\F_q)$. For any automorphism $\sigma \in \Aut(E/\F_q)$ and $P\in \PP_{E\bar{\F}_q}^1$, it is easy to verify that
\[\sigma^{-1}\tau_Q \sigma(P)=\sigma^{-1}(\sigma(P)\oplus Q)=\sigma^{-1}(\sigma(P))\oplus \sigma^{-1}(Q)=P\oplus \sigma^{-1}(Q)=\tau_{\sigma^{-1}(Q)}(P).\]
Hence, $\sigma^{-1}\tau_Q \Gs=\tau_{\sigma^{-1}(Q)}\in T_E$. Moreover, the group law of $\Aut(E/\F_q)$ is given by
$$(\tau_P\Ga)\cdot (\tau_Q\Gb)=\tau_P\cdot \Ga \tau_Q \cdot \Gb=\tau_P\cdot \tau_{\Ga(Q)}\Ga\cdot \Gb=\tau_{P\oplus \Ga(Q)}\cdot \Ga\Gb$$
for any $\tau_P,\tau_Q\in T_E$ and $\Ga,\Gb\in \Aut(E,O)$. This completes the proof.
\end{proof}
\begin{rmk}{\rm The translation group $T_E$ is isomorphic to $\fE(\F_q)\cong \PP_E^1$.   
The stabilizer $\Aut(E,O)$  is a subgroup of $\Aut(\fE)$ and thus its order is a divisor of $24$. This implies that the size of $\Aut(E/\F_q)$ is at most $24|\fE(\F_q)|$.
}\end{rmk}

\subsection{Subgroups of $\Aut(E/\F_q)$}\label{subsec:3.2}
In this subsection, we want to characterize the subgroup structures of automorphism group $\Aut(E/\F_q)$. For any $\sigma \in \Aut(E/\F_q)$, it can be uniquely written as $\sigma=\tau_{\sigma(O)}\rho$ for some automorphism $\rho\in \Aut(E,O)$. Let $\pi: \Aut(E/\F_q)\rightarrow \Aut(E,O)$ be the projection map, i.e., $\pi(\sigma)=\rho$.
From Theorem \ref{thm:3.1}, we obtain the following exact sequence
\[ 1\rightarrow T_E \hookrightarrow \Aut(E/\F_q)\stackrel{\pi} \rightarrow \Aut(E,O) \rightarrow 1.\]
Let $G$ be a subgroup of $\Aut(E/\F_q)$. From the second isomorphism theorem, we have the exact sequence
\[1\rightarrow T_E\cap G \hookrightarrow G \stackrel{\pi} \rightarrow \pi(G) \rightarrow 1.\]
Hence, we have the following group structure of subgroups of automorphism groups of elliptic function fields over finite fields.

\begin{theorem}\label{thm:3.2}
Let $E/\F_q$ be an elliptic function field and let $G$ be a subgroup of $\Aut(E/\F_q)$. Then, we have $G \cong (T_E\cap G) \rtimes \pi(G)$, i.e., every subgroup of $\Aut(E/\F_q)$ is isomorphic to a semiproduct of a subgroup of $T_E$ and a subgroup of $\Aut(E, O)$.
\end{theorem}

\begin{cor}
Let $E/\F_q$ be an elliptic function field and let $G$ be a subgroup of $\Aut(E/\F_q)$. If $\gcd(|T_E|,|\Aut(E,O)|)=1$, then there exist a subgroup $T$ of $T_E$ and a subgroup $A$ of some conjugate of $\Aut(E, O)$ such that $G=TA=T\rtimes A$. 
\end{cor}
\begin{proof}
From Theorem \ref{thm:3.2}, $G\cap T_E$ is a normal subgroup of $G$.  If $\gcd(|T_E|,|\Aut(E,O)|)=1$, then we have $\gcd(|G\cap T_E|, [G:G\cap T_E])=1$. From Schur-Zassenhaus theorem \cite[Theorem 9.1.2]{Ro96}, $G$ contains subgroups of order $[G:G\cap T_E]$ and any two of them are conjugate in $G$. Let $C$ be a subgroup of $G$ with order $[G:G\cap T_E]$. Then we have $G=(G\cap T_E)C=(G\cap T_E)\rtimes C$, i.e., $C$ is a complement of $G\cap T_E$ in $G$. All complements of $T_E$ in $T_E C$ are conjugate, so $C$ is contained in some conjugate $g^{-1}\Aut(E,O) g$ of $\Aut(E,O)$. Hence, $G=(G\cap T_E)\cdot (G\cap g^{-1}\Aut(E,O)g)=(G\cap T_E)\rtimes (G\cap g^{-1}\Aut(E,O)g)$.
\end{proof}

Theorem \ref{thm:3.2} says that for any subgroup $G$ of $\Aut(E/\F_q)$, there exist unique subgroup $T$ of $T_E$ and subgroup $A$ of $\Aut(E, O)$ such that $G\cong T\rtimes A=TA$.  
Conversely, given a subgroup $T$ of $T_E$ and a subgroup $A$ of $\Aut(E,O)$, we want to know whether $TA$ is a subgroup of $\Aut(E/\F_q)$. This question can be answered easily from the following well-known elementary result from group theory.

\begin{lemma}\label{lem:3.2}
Let $A$ and $B$ be two subgroups of a group $H$. Then $AB\le H$ if and only if $AB=BA$.
\end{lemma}
\begin{proof}
If $AB\le G$, then $(AB)^{-1}=AB$. It follows that $$BA=B^{-1}A^{-1}=(AB)^{-1}=AB.$$
Conversely, if $AB=BA$, then $$AB(AB)^{-1}=ABB^{-1}A^{-1}=ABA^{-1}=ABA=AAB=AB.$$
This completes the proof.\end{proof}
To find all subgroups of $\Aut(E/\F_q)$ up to isomorphism,  it follows form Theorem \ref{thm:3.2} and Lemma \ref{lem:3.2} that we need to find all pairs $(T,A)$ with $T\le T_E$ and $A\le \Aut(E, O)$ such that $TA=AT$. However, it is not an easy job to verify $TA=AT$ for a given pair $(T,A)$ with $T\le T_E$ and $A\le \Aut(E, O)$. The following result provides a criterion for which  $TA=AT$.

\begin{theorem}\label{thm:3.4}
Let $T$ be a subgroup of the translation group $T_E$ and let $A$ be a subgroup of $\Aut(E,O)$.
Then $TA$ is a subgroup of $\Aut(E/\F_q)$ if and only if $\tau_{\sigma^{-1}(Q)}\in T$ for all $\sigma\in A$ and $\tau_Q\in T$.
\end{theorem}
\begin{proof}
If $TA$ is a subgroup of $\Aut(E/\F_q)$, then we have $TA=AT$ from Lemma \ref{lem:3.2}. It follows that $\sigma^{-1}T\sigma\subseteq T$ for all $\sigma\in A$. In particular, we have $\sigma^{-1} \tau_Q\sigma=\tau_{\sigma^{-1}(Q)}\in T$ for all $\sigma\in A$ and $\tau_Q\in T$.

If  $\tau_{\sigma^{-1}(Q)}\in T$ for all $\sigma\in A$ and $\tau_Q\in T$, then we have $\sigma^{-1}\tau_Q\sigma=\tau_{\sigma^{-1}(Q)}\in T$.
Then $\sigma^{-1}T\sigma\subseteq T$ for all $\sigma\in A$.  Hence, we have $\sigma T=T\sigma$ for all $\sigma\in A$. Thus, $TA=AT$. From Lemma \ref{lem:3.2}, we have $TA$ is a subgroup of $\Aut(E/\F_q)$.
\end{proof}

\begin{cor}\label{cor:3.5}
If $\PP_E^1$ is cyclic, then for any subgroup $T$ of $T_E$ and any subgroup $A$ of $\Aut(E,O)$,  $TA$ is a subgroup of $\Aut(E/\F_q)$.
\end{cor}
\begin{proof} As $T_E$ is isomorphic to $\PP_E^1$, the translation group $T_E$ is also cyclic.
If $T$ is a subgroup of $T_E$, then there is a rational place $Q$ of $E$ such that $T=\langle \tau_Q\rangle$. Let $t=|T|$. Then $\tau_P\in T$ if and only if $[t]P=O$. As we know $[t]Q=O$, then we have $\sigma^{-1}([t]Q)=\sigma^{-1}(O)=O$ for all automorphisms $\sigma\in \Aut(E,O)$. It follows that $[t]\sigma^{-1}(Q)=O$. Hence, $\tau_{\sigma^{-1}(Q)}\in T$. From Theorem \ref{thm:3.4}, we have $TA$ is a subgroup of $\Aut(E/\F_q)$.
\end{proof}

\begin{cor}\label{cor:3.6}
Let $Q$ be a rational place of the elliptic function field $E/\F_q$ with order $m$, i.e., $[m]Q=O$ and $[k]Q\not= O$ for all $1\le k\le m-1$. Let $\langle Q\rangle$ be the cyclic group generated by $Q$. Let $A$ be the cyclic group generated by an automorphism $\sigma\in \Aut(E, O)$.
Let $T$ be a subgroup of $T_E$ generated by $\rho(Q)$ for all $\rho\in A$.
If $\sigma(Q)\in \langle Q\rangle$, then $TA$ is a subgroup of $\Aut(E/\F_q)$ with order $|TA|=e\cdot m$.
If $\sigma(Q)\notin \langle Q\rangle$, then $TA$ is a subgroup of $\Aut(E/\F_q)$ with order $em^2/|\langle \tau_Q\rangle \cap \langle \tau_{\sigma(Q)}\rangle|\le |TA|\le em^2$.
\end{cor}
\begin{proof}
If $Q$ is a rational place of $E$ with order $m$, then $\sigma(Q)$ is also a rational place of $E$ with order $m$.
If $\sigma(Q)\in \langle Q\rangle$, then $T=\langle \tau_Q\rangle\cong \ZZ_m$. Hence, $TA$ is a subgroup of $\Aut(E/\F_q)$ with order $|TA|=e\cdot m$ from Theorem \ref{thm:3.4}.
If $\sigma(Q)=R\notin \langle Q\rangle$, then $T\supseteq \langle \tau_Q,\tau_R\rangle$ and its order is $m^2/|\langle \tau_Q\rangle \cap \langle \tau_{R}\rangle|\le |T|\le m^2$ from Proposition \ref{prop:2.4}. Hence, $TA$ is a subgroup of $\Aut(E/\F_q)$ with order $em^2/|\langle \tau_Q\rangle \cap \langle \tau_{\sigma(Q)}\rangle|\le |TA|\le em^2$ from Theorem \ref{thm:3.4}.
\end{proof}

Generally, let $Q_i$ be rational places of $E$ with order $m_i\ge 2$ for each $1\le i\le \ell$. Let $A$ be a subgroup of $\Aut(E, O)$.  Let $\mathcal{T}$ be the set $\{\sigma(Q_i): \text{ for all } \sigma\in A \text{ and } 1\le i\le \ell \}$ and let $T$ be the subgroup of $T_E$ generated by $\tau_P$ for all $P\in \mathcal{T}$. Then $TA$ is a subgroup of $\Aut(E/\F_q)$. However, it is difficult to determine the order of $TA$ which depends on the group structure of $\PP_E^1$ or the class group of elliptic function field $E$.

\subsection{Abelian subgroups of $\Aut(E/\F_q)$}\label{subsec:3.3}
Let $E/\F_q$ be an elliptic function field defined over $\F_q$. Let $T$ be a subgroup of $T_E$ and let $A$ be a subgroup of $\Aut(E,O)$.
From Theorem \ref{thm:3.4},  $TA$ is a subgroup of $\Aut(E/\F_q)$ if and only if $\tau_{\sigma^{-1}(Q)}\in T$ for all $\sigma\in A$ and $\tau_Q\in T$.
In order to obtain more subgroups of $\Aut(E/\F_q)$, we need to consider the case $\sigma(Q)\in \langle Q\rangle$ for all $\sigma\in A$ in Corollary \ref{cor:3.6}.
For simplicity, we focus on the case $\sigma(Q)=Q$ in this subsection.
First let us provide an characterization of abelian subgroups of $\Aut(E/\F_q)$.

\begin{theorem}\label{thm:3.7}
Let $E/\F_q$ be an elliptic function field. Then  $G$ be an abelian subgroup of $\Aut(E/\F_q)$ if and only if $\pi(G)$ is abelian and $P\oplus \alpha(Q)=Q\oplus \beta(P)$ for any $\tau_P,\tau_Q\in T_E\cap G$ and any $\alpha,\beta\in \pi(G)$.
\end{theorem}
\begin{proof}
From Theorem \ref{thm:3.1}, every automorphism $\sigma \in \Aut(E/\F_q)$ can be uniquely written as a composition of a translation $\tau_{\sigma(O)}$ and an automorphism $\pi(\sigma)$ of $E$ fixing $O$. For any automorphism $\tau_P\alpha, \tau_Q\beta\in G$, then we have 
$$(\tau_P\alpha)\cdot (\tau_Q\beta)=\tau_{P\oplus \alpha(Q)}\cdot \Ga\Gb \text{ and } (\tau_Q\beta)\cdot (\tau_P\alpha)=\tau_{Q\oplus \Gb(P)}\cdot \Gb\Ga.$$
Hence, $G$ be an abelian group of $\Aut(E/\F_q)$ if and only if $(\tau_P\alpha)\cdot (\tau_Q\beta)=(\tau_Q\beta)\cdot (\tau_P\alpha)$ for any $\tau_P\alpha, \tau_Q\beta\in G$ if and only if $P\oplus \alpha(Q)=Q\oplus \beta(P)$ and $\alpha \beta=\beta\alpha$ for any $\tau_P,\tau_Q\in T_E\cap G$ and any $\alpha,\beta\in \pi(G)$.
\end{proof}

\begin{prop}\label{prop:3.7}
Let $E/\F_q$ be an elliptic function field. Let $Q$ be a rational place of $E$ and $\tau_Q$ be a translation of $E$. Let $\sigma$ be an automorphism in $\Aut(E,O)$.
Let $G$ be a group of $\Aut(E/\F_q)$ generated by $\tau_Q$ and $\sigma$. Then $G=\langle \tau_Q,\sigma\rangle$ is abelian if and only if $\sigma(Q)=Q$.
\end{prop}
\begin{proof}
If $G$ is abelian, then we have $\tau_Q\sigma=\sigma\tau_Q$. In particular, we have $$\sigma(Q)=\sigma\tau_Q(O)=\tau_Q\sigma(O)=\tau_Q(O)=Q.$$
Conversely, if $\sigma(Q)=Q$, then it is easy to verify that $\pi(G)=\langle \Gs\rangle$ is abelian and $P\oplus \alpha(R)=P\oplus R=R\oplus P=R\oplus \beta(P)$ for any $\tau_P,\tau_R\in T_E\cap G=\langle \tau_Q \rangle$ and $\Ga,\Gb\in \langle \Gs\rangle$. From Theorem \ref{thm:3.7}, we have $G=\langle \tau_Q,\sigma\rangle$ is abelian.
\end{proof}

As we know the translation group $T_E$ is an abelian subgroup of the automorphism group $\Aut(E/\F_q)$ of an elliptic function field $E/\F_q$.
In order to find more abelian subgroups and optimal locally repairable codes with more flexible locality in the next section, we need to determine the order of abelian subgroups of automorphism group $\Aut(E/\F_q)$  generated by two elements $\tau_Q\in T_E$ and $\sigma\in \Aut(E,O)$. 

\begin{prop}\label{prop:3.8}
Let $E/\F_q$ be an elliptic function field. Let $Q$ be a rational place of $E$ and $\tau_Q$ be a translation of $E$. Let $\sigma$ be a nontrivial automorphism in $\Aut(E,O)$.  Let $G$ be the abelian group of $\Aut(E/\F_q)$ generated by $\tau_Q$ and $\sigma$. Then the order of $G$ is at most $9$.
\end{prop}
\begin{proof}
From Proposition \ref{prop:3.7},  we have $\sigma(Q)=Q$. Let $m$ be the order of $Q$. Hence, $[k]Q$ is totally ramified in the extension $E/E^{\langle \sigma \rangle}$ for each $0\le k\le m-1$. Let $e\ge 2$ be the ramification index of $Q$ in $E/E^{\langle \sigma \rangle}$. Then we have $e=\text{ord}(\sigma)$.
Consider the Galois extension $E/E^{\langle \sigma \rangle}$, we have
\[0=2g(E)-2=[2g(E^{\langle \sigma \rangle})-2]\cdot \text{ord}(\sigma)+\deg \text{Diff}(E/E^{\langle \sigma \rangle})\ge -2e+m(e-1)\]
from the Hurwitz genus formula and Dedekind's different theorem.

{\bf Case 1:} If $Q$ is tamely ramified in the extension $E/E^{\langle \sigma \rangle}$, then the different exponent $d(Q)$ of $Q$ over $E^{\langle \sigma \rangle}$ is $e-1$ and $$1\le m\le \frac{2e}{e-1}=2+\frac{2}{e-1}\le 4.$$ Thus, we have the following cases.
\begin{itemize}
\item[(1)] If $e=2$, then $m\le 4$. Hence, the order of $G$ is upper bounded by $|G|\le 8$.
\item[(2)] If $e=3$, then $ m\le 3$. Hence, the order of $G$ is upper bounded by $|G|\le 9$.
\item[(3)] If $e\ge 4$, then $m\le 2$. 
If $m=1$, then the order of $G$ is $|G|=e\le 6$ from \cite[Theorem 3.10.1 and Appendix A]{Si86}.
Otherwise $m=2$, we have
\[0=2g(E)-2=-2e+\deg \text{Diff}(E/E^{\langle \sigma \rangle})=-2e+2(e-1)+2.\]
There must exist a rational place $R$ of $E^{\langle \sigma \rangle}$ which splits into two rational places $R_1$ and $R_2$ with ramification index $e(R_i)=2$, since
\begin{eqnarray*}
\sum_{i=1}^r d(R_i)\deg(R_i)&=&\sum_{i=1}^r(e(R_i)-1) f(R_i)\deg R\\ &=&\frac{e(R_i)-1}{e(R_i)}\sum_{i=1}^re(R_i) f(R_i)\deg R\\ &=&\frac{e(R_i)-1}{e(R_i)}\cdot e \cdot\deg(R)=2.
\end{eqnarray*}
Hence, the only possibility is $e=4$ and $|G|= 8$.
\end{itemize}

{\bf Case 2:} If $Q$ is wildly ramified in the extension $E/E^{\langle \sigma \rangle}$, then we have $d(O)=d(Q)\ge e$ and
\[0=2g(E)-2=-2e+\deg \text{Diff}(E/E^{\langle \sigma \rangle})\ge -2e+2e.\]
Hence, $m=2$ and $d(O)=d(Q)=e$. From the Hilbert's different theorem \cite[Theorem 3.8.7]{St09}, we have
$$e=d(O)=\sum_{i=0}^{\infty} (|G_i|-1)=(e-1)+(2-1),$$
where $G_i$ are the $i$-th ramification groups of $O$ in the extension $E/E^{\langle \sigma \rangle}$ for $i\ge 0$.
The above equality holds true only if char$(\F_q)=2$ and $e=2t$ with $2\nmid t$. Thus, we have the following cases.
If $j(E)\neq 0, 1728$, then $|\Aut(E,O)|=2$. Hence, $|G|\le 4$.
If $j(E)=0=1728$, then $|\Aut(E,O)|=24$. In this case, the elliptic curve $E$ can be explicitly given by $y^2+a_3y=x^3+a_4x+a_6$. Any automorphism $\rho\in \Aut(E,O)$ is given in the form $\rho(x)=u^2x+s^2$ and $\rho(y)=y+u^2sx+t$ with $u^3=1$, $s^4+a_3s+(1-u)a_4=0$ and $t^2+a_3t+s^6+a_4s^2=0$. Let $t=x/y$ be a prime element of $O$. Then we have
\begin{eqnarray*}
i(\rho):&=&\nu_O(\rho(t)-t)\\&=&\nu_O\left(\frac{u^2x+s^2}{y+u^2sx+t}-\frac{x}{y}\right)\\
  &=& \nu_O(u^2xy+s^2y+xy+u^2sx^2+tx)-\nu_O(y)-\nu_O(y+u^2sx+t)\\
  &=&  \nu_O(u^2xy+xy+u^2sx^2+s^2y+tx)+6\\
  &=& \begin{cases} 1, \text{ if } u\neq 1,\\ 2, \text{ if } u=1, s\neq 0,\\ 4, \text{ if } u=1,s=0,t\neq 0. \end{cases}
\end{eqnarray*}
From \cite[Proposition 3.5.12 and Theorem 3.8.7]{St09}, we have
\[d(O)=\sum_{1\neq \rho\in \langle \sigma \rangle}i(\rho)=e.\]
There is exactly one automorphism $\rho_1\in \langle \sigma \rangle$ with $\nu_O(\rho_1(t)-t)=2$. In fact, $\rho_1$ must be given by $\rho_1(x)=x+s^2$ and $\rho_1(y)=y+u^2sx+t$ for some $s\neq 0$. Then $\rho_1^2(x)=x, \rho_1^2(y)=y+s^3\neq y$ and $i(\rho_1^2)=2$. Hence, we have $\text{id} \neq \rho_1^2\in \langle \sigma \rangle$ which is impossible. 
From the above discussion, we have shown that the order of $G$ is at most $9$.
\end{proof}

\begin{rmk}
From \cite[Theorem 11.79]{HKT08}, the order of an abelian subgroup $G$ of the automorphism group of function field $F/\F_q$ with genus $g\ge 2$ is upper bounded by
$$|G|\le \begin{cases} 4g+4 \text{ for } p\neq 2,\\ 4g+2 \text{ for } p=2.\end{cases}$$
For a global function field with large genus, the upper bound of abelian subgroups of automorphism group $\Aut(E/\F_q)$ can be sharpened in \cite{MX19}.
In the following subsection \ref{subsec:4.4}, we will provide an explicit abelian subgroup with order $9$ in $\Aut(E/\F_q)$ involving a nontrivial automorphism in $\Aut(E,O)$ for some maximal elliptic function fields.
\end{rmk}

\section{Optimal locally repairable codes via elliptic function fields}\label{sec:4}
In this section, we provide a construction of optimal locally repairable codes via automorphism groups of elliptic function fields over finite fields by generalizing the method given in \cite{JMX20,LMX19,TB14}.

Before presenting the construction of locally repairable codes, let us discuss ramification property of Galois subfields of elliptic function fields.
\subsection{Ramification in Galois subfields of elliptic function fields}\label{subsec:3.4}
In this subsection, we shall consider the ramification information of Galois extension $E/E^{G}$.
The following result shows that there are not many ramified places for an elliptic function field $E/\F_q$.
Hence, there are many rational places of $E$ such that their restriction to $E^{G}$ are splitting completely in  $E/E^{G}$.
\begin{prop} \label{prop:3.9}
Let $E/\F_q$ be an elliptic function field defined over $\F_q$. Let $T$ be a subgroup of the translation group $T_E$ and let $A$ be a nontrivial subgroup of $\Aut(E,O)$ such that $G=T A$ is a subgroup of $\Aut(E/\F_q)$. Let $|G|=r+1$ and let $F=E^G$ be the fixed subfield of $E$ with respect to $G$. Then there are at most $r+1+2|T|$ rational places of $E$ that are ramified in $E/F$. All unramified rational places of $E$ are splitting completely in $E/F$.
\end{prop}
\begin{proof}
Any automorphism $\sigma$ fixes the infinity place $O$ if and only if $\sigma\in \Aut(E,O)$.  Then we have the ramification index $e(O|O\cap F)= |A|$ and the different exponent $d(O|O\cap F)\ge |A|-1$ from Dedekind's different theorem \cite[Theorem 3.5.1]{St09}.  Thus, the genus of $F$ is $0$ from Hurwitz genus formula.
From the fundamental equation for Galois extension, there are $|T|$ ramified rational places with ramification index  $|A|$ and different exponent $d\ge |A|-1$.
Assume that there are $s$ ramified rational places of $E$ in $E/F$.  Each ramified place has different exponent at least $2-1=1$, by the Hurwitz genus formula, we have
\begin{eqnarray*}
0=2g(E)-2 & \ge &  [E:F](2g(F)-2)+(|A|-1)|T|+(s-|T|) \\ &\ge& -(r+1)+s-2|T|.
\end{eqnarray*}
For any unramified rational place $P$ of $E$, the relative degree of $P$ over $P\cap F$ is one, since $f(P|P\cap F)=\deg(P)/\deg(P\cap F)=1$. Hence, all unramified rational places of $E$ are splitting completely in $E/F$. This gives the desired result.
\end{proof}

There are at least $\left\lceil\frac{N(E)-2|T|}{r+1}\right\rceil-1$ rational places of $E^G$ splitting completely in  $E/E^{G}$. For explicit function fields and subgroups of automorphism groups, the upper bound of the number of ramified rational places may be improved such that there are more rational places which are splitting completely in  $E/E^{G}$.

\subsection{A general construction via elliptic function fields}\label{subsec:4.1}
In this subsection, let us provide a general construction of locally repairable codes via automorphism groups of elliptic function fields over finite fields \cite{LMX19}.

Let  $E/\F_q$ be an elliptic function field. Let $G$ be a subgroup of automorphism group $\Aut(E/\F_q)$ of order $r+1$ and let $F=E^{G}$ be the fixed subfield of $E$ with respect to $G$. From Galois theory, $E/F$ is a Galois extension with Galois group $\Gal(E/F)=G.$
Assume rational places $Q_1,\cdots, Q_\ell$ of $F$ are splitting completely in $E/F$ for each $1\le i \le \ell$. Let $P_{i,1}, P_{i,2},\cdots, P_{i,r+1}$ be the rational places of $E$ lying over $Q_i$ and let $\mP=\{P_{i,j}: 1\le i \le \ell, 1\le j\le r+1\}$. Then the cardinality of $\mP$ is $\ell(r+1)$.

Choose a divisor $D$ of $F$ such that $\text{supp}(D) \cap \{Q_1,\cdots,Q_m\}=\emptyset$. The Riemann-Roch space $\mathcal{L}(D)=\{f\in F^*: (f)\ge -D\}\cup \{0\}$
 is a finite-dimensional vector space over $\F_q$. Let $\{z_1,\cdots, z_t\}$ be a basis of  $\mathcal{L}(D)$ over $\F_q$. Choose elements $w_i\in E$ such that $w_0=1,w_1,\cdots,w_{r-1}$ are linearly independent over $F$ and $\nu_{P_{i,j}}(w_k)\ge 0$ for all $1\le i\le \ell$, $1\le j\le r+1$ and $0\le k\le r-1$. Consider the set of functions
\begin{equation}\label{eq:12} V:=\left\{\sum_{j=1}^t a_{0j}z_j+\sum_{i=1}^{r-1} \left(\sum_{j=1}^{t-1} a_{ij}z_j\right)w_i\in E:\; a_{ij}\in \F_q \right\}.\end{equation}

\begin{lemma}\label{lem:4.1}
Let $i$ be an integer between $1$ and $\ell$, and suppose that every $r\times r$ submatrix of the matrix
$$M=\left(\begin{array}{cccc}w_0(P_{i,1}) & w_1(P_{i,1}) & \cdots & w_{r-1}(P_{i,1}) \\w_0(P_{i,2}) & w_1(P_{i,2})  & \cdots &w_{r-1}(P_{i,2})  \\\vdots & \vdots & \ddots & \vdots \\w_0(P_{i,r+1}) & w_1(P_{i,r+1})  & \cdots &w_{r-1}(P_{i,r+1})\end{array}\right)$$
is invertible. Then the value of $f\in V$ at any place in the set $\{P_{i,1}, P_{i,2},\cdots, P_{i,r+1}\}$ can be recovered from the values of $f$ at the other $r$ places.
\end{lemma}
\begin{proof}
Please refer to \cite[Proposition 2]{BHHMV17} or \cite[Proposition 18]{LMX19}.
\end{proof}

\begin{lemma}\label{lem:4.2}
Let $\mP$ and $V$ be defined as above and satisfy the assumption of Lemma {\rm \ref{lem:4.1}}. If $V$ is contained in $\mL(D^\prime)$ for a divisor $D^\prime$ of $E$ with $\deg(D^\prime)< \ell (r+1)$ and $\supp(D^\prime)\cap\{P_{i,1},\dots,P_{i,r+1}\}_{i=1}^{\ell}=\emptyset$, then the algebraic geometry code
$$C(\mP,V)=\{(f(P))_{P\in \mP}: f\in V\}$$ is a $q$-ary $[n,k,d]$-locally repairable code with locality $r$, length $n=\ell(r+1)$, dimension $k=rt-(r-1)$ and minimum distance $d\ge n-\deg(D^\prime)$.
\end{lemma}
\begin{proof}
The locality property follows from Lemma \ref{lem:4.1}. For more details, please refer to \cite[Proposition 19]{LMX19}.
\end{proof}

By considering subgroups of the automorphism group $\Aut(E/\F_q)$, we can choose a vector space $V$ and a set of rational places $\mP$ such that the assumption of Lemma \ref{lem:4.1} is satisfied. The following result is a generalization of \cite[Proposition 20]{LMX19}.

\begin{prop}\label{prop:4.3}
Let $E/\F_q$ be an elliptic function field. Let $T$ be a subgroup of the translation group $T_E$ and let $A$ be a nontrivial subgroup of $\Aut(E,O)$ such that $G=TA$ is a subgroup of $\Aut(E/\F_q)$. Let $|G|=(r+1)\le q$ and $|A|\ge 2$. Let $F=E^G$ and let $P$ be a rational place of $E$ such that $P\cap F$ is splitting completely into $\{P_1=P,P_2,\cdots, P_{r+1}\}$ in the extension $E/F$. Then
\begin{itemize}
\item[{\rm (i)}] there exists an element  $z\in E$ satisfying that $F=\F_q(z)$ and $(z)_\infty=\sum_{i=1}^{r+1}P_i$;
\item[{\rm (ii)}] there exist elements $w_i\in E$ with $(w_i)_\infty=P_1+P_2+\cdots+P_{i+1}$ for $0\le i\le r-1$  such that they are linearly independent over $F$;
\item[{\rm (iii)}] let  $\{P_{i,1},P_{i,2},\cdots,P_{i,r+1}\}$ be the pairwise distinct rational places lying over the rational place $Q_i$ of $F$ for each $1\le i\le \ell$, such that
$\{P_{i,1},P_{i,2},\cdots,P_{i,r+1}\}_{i=1}^l \cap \{P_{1},P_{2},\cdots,P_{r+1}\}=\emptyset.$
Then every $r\times r$ submatrix of the matrix
$$M=\left(\begin{array}{cccc}1 & w_1(P_{i,1}) & \cdots & w_{r-1}(P_{i,1}) \\1 & w_1(P_{i,2})  & \cdots &w_{r-1}(P_{i,2})  \\\vdots & \vdots & \ddots & \vdots \\1& w_1(P_{i,r+1})  & \cdots &w_{r-1}(P_{i,r+1})\end{array}\right)$$
is invertible for all $1\le i\le \ell$.
\end{itemize}
\end{prop}

\begin{proof}
(i) Since $P\cap F$ is splitting completely in $E/F$, $P\cap E^{A}$ is a rational place of $E^{A}$ and splitting completely in $E$. Let $\Gs_1=1, \Gs_2, \cdots, \Gs_{|A|}$ be the automorphisms of $E$ with the pairwise distinct places $\Gs_i(P)=P_i$. As we have assumed that $|A|\ge 2$, then $O$ is ramified in $E/E^{A}$  and the genus of $E^{A}$ is $0$ from Hurwitz genus formula \cite[Lemma 6]{LMX19}. Hence, there exists an element $z_0\in E^{A}$ such that $(z_0)=O\cap E^{A}- P\cap E^{A}$ as a divisor of $E^{A}$ and the principal divisor of $z_0$ in $E$ is
$$(z_0)=|A|O-P_1-P_2-\cdots-P_{|A|}.$$
For any automorphism $\tau\in T$, then we have $(z_0)^{\tau}=|A|\tau O-\tau(P_1)-\tau(P_2)-\cdots-\tau(P_{|A|}).$
Hence, we have \[ \sum_{\tau\in T} (z_0)^{\tau}=|A|\sum_{\tau\in T} \tau O-\sum_{\tau\in T} \sum_{j=1}^{|A|}\tau(P_j)=|A|\sum_{\tau\in T} \tau O-\sum_{\sigma\in G} \sigma(P)=|A|\sum_{\tau\in T} \tau O-\sum_{j=1}^{r+1}P_j.\]
From \cite[Lemma 11.4]{HKT08}, we have $(z_0)^{\tau}=(\tau^{-1}(z_0))$. Thus, there exists an element $z=\prod_{\tau\in T}\tau^{-1}(z_0) \in E^G$ such that
\[(z)=|A|\sum_{\tau\in T}\tau O-P_1-P_2-\cdots-P_{r+1}.\]
It is easy to see that $r+1=|G|=[E:E^G]\le [E:\F_q(z)]=\deg(z)_\infty=r+1$ from \cite[Theorem 1.4.11]{St09}. Hence, we obtain $E^G=\F_q(z).$

(ii)  As the dimension of Riemann-Roch space $\mL(P_1+P_2)$ is $\dim_{\F_q} \mL(P_1+P_2)=\deg(P_1+P_2)-g(E)+1=2$ from the Riemann-Roch Theorem, there exists an element $w_1\in \mL(P_1+P_2)\setminus \F_q$ such that $(w_1)_{\infty}=P_1+P_2$.
For each $2\le i\le r-1$, the set $\cup_{j=1}^{i+1}\mL(\sum_{u=1}^{i+1}P_u-P_j)$ has size at most $(i+1)q^{i}$ which is less than $q^{i+1}=|\mL(\sum_{u=1}^{i+1}P_u)|$. Hence, there exists an element $w_i\in E$ such that $(w_i)_{\infty}=P_1+P_2+\cdots+P_{i+1}$ for each $1\le i\le r-1$. Moreover, it is easy to verify that $w_0=1, w_1, \cdots, w_{r-1}$ are linearly independent over $\F_q(z)$ from the strictly triangle inequality \cite[Lemma 1.1.11]{St09} or  the proof of \cite[Proposition 20]{LMX19}.

(iii) If the pairwise distinct places $P_{i,1},\cdots,P_{i,r+1}$ lie over the same rational place $z-\Gb_i$ of $F$ for some $\Gb_i\in\F_q^*$, we claim
every $r\times r$ submatrix of the matrix
$$M=\left(\begin{array}{cccc}1 & w_1(P_{i,1}) & \cdots & w_{r-1}(P_{i,1}) \\1 & w_1(P_{i,2})  & \cdots &w_{r-1}(P_{i,2})  \\\vdots & \vdots & \ddots & \vdots \\1& w_1(P_{i,r+1})  & \cdots &w_{r-1}(P_{i,r+1})\end{array}\right)$$
is invertible. Without loss of generality, we consider the first $r$ rows. Suppose that
$$\text{det}\left(\begin{array}{cccc}1 & w_1(P_{i,1}) & \cdots & w_{r-1}(P_{i,1}) \\1 & w_1(P_{i,2})  & \cdots &w_{r-1}(P_{i,2})  \\\vdots & \vdots & \ddots & \vdots \\1& w_1(P_{i,r})  & \cdots &w_{r-1}(P_{i,r})\end{array}\right)=\bo.$$
Then there exists $(c_0,\cdots,c_{r-1})\in \F_q^r\setminus\{\bo\}$  such that
$$\left(\begin{array}{cccc}1 & w_1(P_{i,1}) & \cdots & w_{r-1}(P_{i,1}) \\1 & w_1(P_{i,2})  & \cdots &w_{r-1}(P_{i,2})  \\\vdots & \vdots & \ddots & \vdots \\1& w_1(P_{i,r})  & \cdots &w_{r-1}(P_{i,r})\end{array}\right) \left(\begin{array}{c} c_0 \\  c_1\\  \vdots \\ c_{r-1}\end{array} \right)=\bo.$$
Then we have $(c_0+c_1w_1+\cdots+c_{r-1}w_{r-1})(P_{i,j})=0$ for all $1\le j\le r$ and hence $c_0+c_1w_1+\cdots+c_{r-1}w_{r-1}\in \mL(P_1+\cdots+P_r-P_{i,1}-\cdots-P_{i,r})$. Thus, the principal divisor of  $c_0+c_1w_1+\cdots+c_{r-1}w_{r-1}$ is
$$(c_0+c_1w_1+\cdots+c_{r-1}w_{r-1})=\sum_{j=1}^r P_{i,j}-\sum_{j=1}^r P_{j}.$$
As the places $P_{i,1},\cdots,P_{i,r}$ lie over the same rational place of $z-\Gb_i$, the equation
$$z=\prod_{\tau\in T_0}\tau^{-1}(z_0)=\frac{h(z_0)}{g(z_0)}\equiv \Gb_i \ (\text{mod } P_{z-\Gb_i})$$
has $|T|$ distinct nonzero roots $z_0=\Ga_{i,1}, \Ga_{i,2},\cdots, \Ga_{i,|T|}$. After rearranging the order of places $P_{i,j}$ for $1\le j\le r+1$, we may assume that
$$(z_0^{-1}-\Ga_{i,j}^{-1})=P_{i,(j-1)|A|+1}+\cdots+P_{i,j|A|}-|A|O \text{ for } 1\le j\le |T|.$$
Thus, we have $$P_{r+1}-P_{i,r+1}=\Big{(}(c_0+c_1w_1+\cdots+c_{r-1}w_{r-1})\cdot z^{-1} \cdot \prod_{j=1}^{|T|}\frac{1}{z_0^{-1}-\Ga_{i,j}^{-1}}\Big{)}.$$
This is a contradiction by Lemma \ref{lem:2.1}.
\end{proof}

With the above preparation, we can obtain the following explicit construction of optimal locally repairable codes from elliptic function fields over finite fields.

\begin{prop}\label{prop:4.4}
Let $E/\F_q$ be an elliptic function field with $N(E)$ rational places. Let $T$ be a subgroup of the translation group $T_E$ and let $A$ be a nontrivial subgroup of $\Aut(E,O)$ such that $G=TA$ is a subgroup of $\Aut(E/\F_q)$. Let $|G|=|T|\cdot |A|=(r+1)$ and $|A|\ge 2$.
Let $F=E^G$ be the fixed subfield of $E$ with respect to $G$. Then there exists an optimal $q$-ary $[n=m(r+1),k=rt-r+1, d=n-(t-1)(r+1)]$ locally repairable code with locality $r$ for any $1\le t \le m\le \ell= \left\lceil\frac{N(E)-2|T|}{r+1}\right\rceil -1$.
\end{prop}

\begin{proof}  By Proposition \ref{prop:3.9}, there are at most $r+1+2|T|$ rational places that are ramified in $E/F$ and all unramified rational places of $E$ are splitting completely in $E/F$. Hence, there are $\ell-1$ sets $\{P_{i,1},\dots,P_{i,r+1}\}_{i=1}^{\ell-1}$ that do not intersect with ramified points and $r+1$ poles of $z$. Put $z_i=z^{i-1}$ for $i=1,2,\dots,t$ and consider the set $V$ of functions given in \eqref{eq:12}.
Then $V$ is a subspace of $\mL((t-1)(P_1+\cdots+P_{r+1})),$ where $P_1,P_2,\dots,P_{r+1}$ are $r+1$ pole places of $z$ given in the proof of Proposition \ref{prop:4.3}.
Let $1\le t\le m\le \ell -1$ and $\mP=\cup_{i=1}^m \{P_{i,1},\dots,P_{i,r+1}\}$. By Lemma \ref{lem:4.2}, the algebraic geometry code
$C(\mP,V)=\{(f(P))_{P\in \mP}: f\in V\}$ is an $[n=m(r+1),k=rt-r+1,d\ge n-(t-1)(r+1)]$ locally repairable codes with locality $r$. On the other hand, by the Singleton-type bound \eqref{eq:x1},
\[d\le n-k- \left\lceil \frac kr\right\rceil+2=n-rt+r-1-\left\lceil \frac {tr-r+1}r\right\rceil+2=n-(t-1)(r+1).\]
Hence, the code $C(\mP,V)$ is optimal.

Using the modified algebraic geometry codes, we can include the poles of $z$ in the set of evaluation points.
The modified algebraic geometry code defined by
$C=\{(f(P_{1,1}), \cdots, f(P_{1,r+1}), \cdots, f(P_{\ell-1,1}), \cdots, f(P_{\ell-1,r+1}), (z^{1-t}f)(P_1),\cdots,(z^{1-t}f)(P_{r+1}))\\  :  f\in V\}$
is an optimal  $[n=\ell (r+1),k=rt-r+1,d= n-(t-1)(r+1)]$ locally repairable code with locality $r$.
It remains to prove that the locality property holds true at rational places $P_i$ for $1\le i\le r+1$.
Let $f=\sum_{i=0}^{r-1}f_i(z)w_i=\sum_{j=0}^{t-1}a_{0,j}z^j+\sum_{j=0}^{t-2}a_{1,j}z^jw_1+\cdots+\sum_{j=0}^{t-2}a_{r-1,j}z^jw_{r-1}\in V$. Then we have
\[(z^{1-t}f)(P_i)=a_{0,t}+a_{1,t-1}\left(\frac{w_1}{z}\right)(P_i)+\cdots+a_{r-1,t-1}\left(\frac{w_{r-1}}{z}\right)(P_i).\]
It will be sufficient to prove that any $r\times r$ submatrix of matrix
$$M=\left(\begin{array}{cccc}1 & \left(\frac{w_1}{z}\right)(P_1) & \cdots & \left(\frac{w_{r-1}}{z}\right)(P_1) \\1 & \left(\frac{w_1}{z}\right)(P_2) & \cdots & \left(\frac{w_{r-1}}{z}\right)(P_2) \\\vdots & \vdots & \ddots & \vdots \\ 1& \left(\frac{w_1}{z}\right)(P_{r+1})  & \cdots & \left(\frac{w_{r-1}}{z}\right)(P_{r+1})\end{array}\right)$$
is invertible. Suppose that the $1,\cdots,i_0-1,i_0+1,\cdots,r+1$ rows are linearly dependent. Let $A$ be the matrix which is obtained from $M$ by deleting the $i_0$-th row. Then there exists an vector $(c_0,c_1,\cdots,c_{r-1})\in \F_q^r\setminus \{\bo\}$ such that $A(c_0,c_1,\cdots,c_{r-1})^T=\bo$.
That is to say that \[\left(c_0+c_1\frac{w_1}{z}+\cdots+c_{r-1}\frac{w_{r-1}}{z}\right)(P_i)=0 \text{ for } 1\le i\neq i_0\le r+1.\]
It is easy to verify that \[c_0+c_1\frac{w_1}{z}+\cdots+c_{r-1}\frac{w_{r-1}}{z}\in \mL\left(|A|\sum_{\tau\in T} \tau(O)+P_{i_0}-\sum_{i=1}^{r+1}P_i\right).\]
Hence, the principal divisor of $h=c_0+c_1\frac{w_1}{z}+\cdots+c_{r-1}\frac{w_{r-1}}{z}$ is given by
\[(h)=\sum_{i=1}^{r+1}P_i-P_{i_0}-|A|\sum_{\tau\in T} \tau(O)+P\] for some rational place $P$ of $E$.
Furthermore, we have
$$(z)=|A|\sum_{\tau\in T} \tau(O)-\sum_{j=1}^{r+1}P_j.$$
It follows that the principal divisor of $hz$ is \[(hz)=(c_0z+c_1w_1+\cdots+c_{r-1}w_{r-1})=P-P_{i_0}.\]
If $P\neq P_{i_0}$, then this is impossible for elliptic function fields from Lemma \ref{lem:2.1}.
Otherwise $P=P_{i_0}$, then $c_0z+c_1w_1+\cdots+c_{r-1}w_{r-1}=a\in \F_q^*.$ This contradicts to the fact that $1,w_1,\cdots, w_{r-1}$ are linearly independent over $\F_q(z)$.
\end{proof}

\subsection{Locality $r=2|T|-1$ with $|A|=2$}\label{subsec:4.2}
From Theorem 3.10.1 and Appendix A in \cite{Si86}, there is an automorphism $\Gs\in \Aut(E,O)$ of the elliptic function field $E/\F_q$ with order $2$ which can be given by $\sigma(x)=x$ and
$$\sigma(y)=\begin{cases}
-y,  & \text{char}(\F_q)\neq 2,\\
y+1, &  \text{char}(\F_q)=2 \text{ and } j(E)=0,\\
y+x, & \text{char}(\F_q)=2 \text{ and } j(E)\neq 0.
\end{cases}$$

\begin{prop}\label{prop:4.5}
Let $E/\F_q$ be an elliptic function field over $\F_q$.
Let $T$ be any subgroup of the translation group $T_E$ and let $A$ be a cyclic subgroup of $\Aut(E,O)$ generated by an automorphism $\sigma$ defined as above.
Then $TA$ is a subgroup of $\Aut(E/\F_q)$ with order $2|T|$.
\end{prop}
\begin{proof}
From Theorem 3.10.1 and Appendix A in \cite{Si86}, $\sigma$ is indeed an automorphism of $E$ fixing the infinity place $O$.
It is easy to verify that $\sigma(P)=-P$ for any rational place $P$ of $E$ from Group Law Algorithm 2.3 given in \cite[Chapter 3]{Si86}.
For every translation $\tau_Q\in T$, we have $\tau_{\sigma^{-1}(Q)}=\tau_{-Q}\in T.$
Hence, $TA$ is a subgroup of $\Aut(E/\F_q)$ with order $2|T|$ from Theorem \ref{thm:3.4}.
\end{proof}

From Proposition \ref{prop:4.5}, we can determine all the orders of subgroups of automorphism group $\Aut(E/\F_q)$ with $|A|=2$.

\begin{prop}\label{prop:4.6}
Let $E/\F_q$ be an elliptic function field with $N(E)$ rational places. For each divisor $|T|$ of $N(E)$, let $r=2|T|-1$ and $1\le m\le \left\lceil \frac{N(E)-2|T|}{r+1}\right\rceil -1=\left\lceil \frac{N(E)}{r+1}\right\rceil-2$. For each $1\le t\le m$, then there exists an optimal $q$-ary $[n=m(r+1), k=rt-r+1,d=n-(t-1)(r+1)]$ locally repairable code with locality $r$.
\end{prop}
\begin{proof}
As $T_E$ is an abelian group of order $N(E)$, for any divisor $|T|$ of $N(E)$, there is a subgroup $T$ of translation group $T_E$ with order $|T|$.
Hence, this theorem follows immediately from Proposition \ref{prop:4.4} and Proposition \ref{prop:4.5}.
\end{proof}

Given a fixed finite field $\F_q$ and an integer $N$, is there an elliptic function field defined over $\F_q$ with $N$ rational places?
This problem was completely solved in \cite{Wa69, Sc87}. In order to obtain long optimal locally repairable codes, we focus on maximal elliptic function fields.

\begin{theorem}\label{thm:4.7}
Let $q=p^a$ for any prime $p$ and any even integer $a>0$. For any positive divisor $h$ of $(\sqrt{q}+1)^2$,  then there exists an optimal $q$-ary $[n=m(r+1), k=r(t-1)+1,d=n-(t-1)(r+1)]$ locally repairable code with locality $r=2h-1$ for any integers $t$ and $m$ satisfying $1\le t<m\le\left\lceil \frac{q+2\sqrt{q}-2r-1}{r+1}\right\rceil$.
\end{theorem}
\begin{proof}
Since $q$ is a prime power with even order, there exists a maximal elliptic function field over $\F_q$ from Lemma \ref{lem:2.3}.
Now this theorem follows immediately from Proposition \ref{prop:4.6}.
\end{proof}

\subsection{Locality $r=|A|\cdot |T|-1$ with $|A|\ge 3$}\label{subsec:4.3}
If $|A|\ge 3$, then the order of subgroup $TA$ of $\Aut(E/\F_q)$ depends on the group structure of rational places of elliptic function fields from Theorem \ref{thm:3.4} and Corollary \ref{cor:3.6}.

\begin{prop}\label{prop:4.8}
Let $E/\F_q$ be an elliptic function field with a cyclic translation group $T_E$ of order $N(E)$, i.e., $T_E\cong \ZZ/(N(E))$.
For any divisor $h$ of $N(E)$ and any subgroup $A$ of $\Aut(E,O)$,  let $r=h|A|-1$ and $1\le m\le \left\lceil \frac{N(E)-2h}{r+1}\right\rceil -1.$
For each $1\le t\le m$, then there exists an optimal $q$-ary $[n=m(r+1), k=rt-r+1,d=n-(t-1)(r+1)]$ locally repairable code with locality $r$.
\end{prop}
\begin{proof}
This result follows immediately from Corollary \ref{cor:3.5} and Proposition \ref{prop:4.4}.
\end{proof}

\begin{prop}\label{prop:4.9}
Let $E/\F_q$ be an elliptic function field with $N(E)$ rational places, where $N(E)=\prod_\ell \ell^{h_\ell}.$ Let $A$ be a subgroup of $\Aut(E,O)$. For any $0\le a_\ell \le \min\{\nu_\ell(q-1),[h_\ell/2]\}$ for $\ell\neq p$ and $0\le a_p\le h_p$, let $|T|=p^{a_p} \prod_{\ell\neq p} \ell ^{2a_\ell}$, $r=|A|\cdot |T|-1$ and $1\le m\le \left\lceil \frac{N(E)-2|T|}{r+1}\right\rceil -1.$ For each $1\le t\le m$, then there exists an optimal $q$-ary $[n=m(r+1), k=rt-r+1,d=n-(t-1)(r+1)]$ locally repairable code with locality $r$.
\end{prop}
\begin{proof}
For any $0\le a_\ell \le \min\{\nu_\ell(q-1),[h_\ell/2]\}$ for $\ell\neq p$ and $0\le a_p\le h_p$, there is a subgroup $T$ of $T_E$ with order $|T|=p^{a_p} \prod_{\ell\neq p} \ell ^{2a_\ell}$ such that $TA$ is a subgroup of $\Aut(E/\F_q)$ from Proposition \ref{prop:2.4} and Corollary \ref{cor:3.6}.
The remaining part follows from Proposition \ref{prop:4.4}.
\end{proof}

In particular, we have the following results on optimal locally repairable codes for maximal elliptic function fields. 
\begin{theorem}\label{thm:4.10}
Let $q=p^a$ for any prime $p$ and any even integer $a>0$. Let $E/\F_q$ be a maximal elliptic function field. Let $A$ be a subgroup of $\Aut(E,O)$. For any positive divisor $h$ of $\sqrt{q}+1$,  then there exists an optimal $q$-ary $[n=m(r+1), k=r(t-1)+1,d=n-(t-1)(r+1)]$ locally repairable code with locality $r=h^2|A|-1$ for any integers $t$ and $m$ satisfying $1\le t<m\le\left\lceil \frac{q+2\sqrt{q}-2h^2-r}{r+1}\right\rceil$, provided that $|A|$ and $p$ satisfy one of the following cases:
\begin{itemize}
\item[{\rm (i)}] $|A|=2,3,4,6,8,12,24$ for $p=2$;
\item[{\rm (ii)}] $|A|=2,3,4,6,12$ for $p=3$;
\item[{\rm (iii)}] $|A|=2,3,6$ for $p\equiv 2(\text{mod } 3)$ and $p\neq 2$;
\item[{\rm (iv)}] $|A|=2,4$ for $p\equiv 3(\text{mod } 4)$  and $p\neq 3$.
\end{itemize}
\end{theorem}
\begin{proof}
If $E/\F_q$ is a maximal elliptic function field, then the group structure of the translation group $T_E$ of $E$ is given by
$$T_E\cong \ZZ/(\sqrt{q}+1)\times \ZZ/(\sqrt{q}+1).$$
Hence, there exists an subgroup of order $h^2|A|$ from Corollary \ref{cor:3.6}.
This result follows immediately from Proposition \ref{prop:4.4} and Lemmas \ref{lem:2.5}, \ref{lem:2.6}, \ref{lem:2.7} and \ref{lem:2.8}.
\end{proof}

\begin{rmk}
\begin{itemize}
\item[(1)] If $h=1$, then $r=|A|-1$ and Theorem \ref{thm:4.10} is the same as the main result of Theorem 2 in \cite{LMX19}. From the above theorem, we have shown that there are optimal locally repairable codes with more flexible localities compared with the results given in \cite{LMX19}.
\item[(2)] The length of optimal locally repairable codes may be improved for explicit examples, since the number of rational places which are splitting completely in $E/F$ may be improved from Proposition \ref{prop:3.9}.
\item[(3)] For some divisor $h$ of $\sqrt{q}+1$, there may exist an subgroup of $\Aut(E/\F_q)$ with order $h|A|$ from Corollary \ref{cor:3.6}.
We will provide an explicit example of order $h|A|=3\times 3$ in the following subsection.
\end{itemize}
\end{rmk}

\subsection{Optimal locally repairable codes with locality $r=8$}\label{subsec:4.4}
In this subsection, we provide an explicit abelian subgroup with order $9$ of automorphism groups of maximal elliptic function fields and hence obtain an explicit construction of optimal locally repairable codes with locality $r=8$ via elliptic function fields.

Let $q$ be an odd power of $4$, i.e., $q=4^{2a+1}$ for a non-negative integer $a\in \ZZ$. Consider the elliptic function field $E=\F_q(x,y)$ defined by the equation $y^2+y=x^3$. From \cite[Lemma 15]{LMX19}, $E/\F_q$ is a maximal elliptic curve, i.e.,  the number of rational places of $E$ is
$N(E)=q+2\sqrt{q}+1=(2^{2a+1}+1)^2.$

Let $Q=(0,1)$ be a rational place of $E$. Consider the translation-by-$Q$ on the elliptic curve $E$ given by
\begin{eqnarray*}
 \tau_Q: &E& \rightarrow E \\
  &P&\mapsto P\oplus Q
\end{eqnarray*}
From Group Law Algorithm 2.3 in \cite{Si86}, we have $x(P\oplus Q)=\frac{y+1}{x^2}$ and $y(P\oplus Q)=\frac{y+1}{y}$.
The translation-by-$Q$ induces an automorphism of elliptic function field $E$, which is still denoted as $\tau_Q$ and given by
\[\tau_Q: \begin{cases} x\mapsto \frac{y+1}{x^2},\\ y \mapsto \frac{y+1}{y}.\end{cases}\]
It is easy to see that the order of $\tau_Q$ is $3$, since we have
\[x\mapsto \frac{y+1}{x^2} \mapsto \frac{x}{y+1} \mapsto x \text{ and } y \mapsto \frac{y+1}{y} \mapsto \frac{1}{y+1}  \mapsto y.\]
From Appendix A in \cite{Si86}, any automorphism $\sigma\in \Aut(E,O)$ is given in the following explicit form
\[ \begin{cases} \sigma(x)=u^2x+s^2,\\ \sigma(y)=y+u^2sx+t,\end{cases}\]
where $u,s,t\in \F_q$ satisfy $u^3=1, s^4+s=0, t^2+t+s^6=0.$
In the following, we fix an automorphism $\sigma\in \Aut(E,O)$ of order $3$ which is given by $\sigma(x)=u^2x, \sigma(y)=y$. Here $u$ is fixed as a primitive third root of root in $\F_q$.
Let $G$ be the subgroup of  $\Aut(E/\F_q)$ generated by $\tau_Q$ and $\sigma$.
It is easy to verify that $\Gs(Q)=Q$. 
Then $G$ is an abelian group of order $9$ from Proposition \ref{prop:3.7} and Proposition \ref{prop:3.8}.
Hence, we have $$G=\langle \tau_Q, \sigma|\tau_Q^3=1=\sigma^3, \sigma \tau_Q=\tau_Q \sigma\rangle \cong \ZZ_3\times \ZZ_3.$$

Now let us determine the ramification information of $E/E^G$.
Since the characteristic of $\F_q$ is two and $|G|=9$, the extension $E/E^G$ is tamely ramified.
Any automorphism $\sigma\in \Aut(E/\F_q)$ fixes $O$ if and only if $\sigma\in \Aut(E,O)$.
Then we have the ramification index $e(O|O\cap E^G)=|G \cap \Aut(E,O)|= 3$ and different exponent $d(O|O\cap E^G)=e(O|O\cap E^G)-1=2$.
From the Hurwitz genus formula $2g(E)-2=|G|\cdot (2g(E^G)-2)+\deg\text{Diff}(E/E^G),$
we have $g(E^G)=0$ and $\deg\text{Diff}(E/E^G)=18$. Moreover, any ramified place has an different exponent $2$ or $8$ in $E/E^G$.
Hence, there are at most $9$ rational places of $E$ which are ramified in the extension $E/E^G$.
Since all unramified rational places are splitting completely in $E/E^G$ and $9|N(E)$,
 there are exactly nine ramified rational places of $E$ which have the different exponent $2$ in $E/E^G$.

Let $z$ be an element in $E^G$ given by \[z=\frac{1}{\sum_{\sigma\in G} \sigma(y)}=\frac{y(y+1)}{y^3+y+1}\in E^G.\]
Then the principal divisor of $z$ is
\[(z)=-\left(\frac{y^3+y+1}{y(y+1)}\right)=3P_\infty+3P_{0,0}+3P_{0,1}-\sum_{j=1}^9P_j,\]
where $P_j$ are zero places of $y^3+y+1$ in $E$  for $1\le j\le 9$.
It is easy to see that $E^G=\F_q(z)$, since $[E:\F_q(z)]=\deg(z)_{\infty}=9=|G|=[E:E^G]$.
Choose elements $w_i\in E$ with $(w_i)_{\infty}=P_1+P_2+\cdots+P_{i+1}$ for each $1\le i\le r-1$ such that $w_0=1, w_1, \cdots, w_{r-1}$ are linearly independent over $\F_q(z)$ from Proposition \ref{prop:4.3}. Let $t$ be a positive integer and let $V_t$ be a vector space over $\F_q$ defined by
\[V_t=\left\{\sum_{i=0}^7 f_i(z)w_i| \deg f_0(z)\le t-1, \deg f_i(z)\le t-2 \text{ for } 1\le i\le 7\right\}\subseteq \mL((t-1)(z)_\infty).\]
There are exactly $\ell:=(q+2\sqrt{q}-8)/9$ rational places of $E^G$ which are splitting completely in $E/E^G$.
One is the infinity place of $\F_q(z)$, the other are $Q_1,Q_2,\cdots, Q_{\ell-1}$.
Let $P_{i,j}$ be the rational places of $E$ lying over $Q_i$ for each $1\le i\le \ell-1$.
For $1\le t\le m\le \ell$, the modified algebraic geometry code
$C=\{(f(P_{1,1}), \cdots, f(P_{1,9}), \cdots, f(P_{\ell-1,1}), \cdots, \\ f(P_{\ell-1,9}), (z^{1-t}f)(P_1),\cdots,(z^{1-t}f)(P_9))| f\in V_t\}$
is an optimal $[9m, 8t-7, 9m-9t+9]$ locally repairable code with locality $r=8$ from Proposition \ref{prop:4.4}.

\begin{theorem}\label{thm:4.11}
 Let $q$ be an odd power of $4$, i.e., $q=4^{2a+1}$ for a non-negative integer $a\in \ZZ$.
 For $1\le t\le m\le \frac{q+2\sqrt{q}-8}{9}$, there is an optimal locally repairable code with length $n=9m\le q+2\sqrt{q}-r$, dimension $k=8t-7$, distance $n-9t+9$ and locality $r=8$.
\end{theorem}

\end{document}